\documentclass[11pt]{article}
\usepackage{amsmath, amsfonts, amssymb}
\usepackage{graphicx, amsthm,color}
\usepackage{dsfont, url}

\newtheorem{theorem}{Theorem}[section]

\newtheorem{proposition}[theorem]{Proposition}

\theoremstyle{definition}

\newtheorem{example}[theorem]{Example}
\theoremstyle{remark}
\newtheorem{remark}[theorem]{Remark}
\numberwithin{equation}{section}

\newcommand{\R}{\mathbb{R}}

\newcommand{\N}{\mathbb{N}}

\newcommand{\norm}[1]{\lVert#1\rVert_0}
\providecommand{\tto}{\mathop{\rightrightarrows}\nolimits}

\newcommand{\ind}{\mathds{1}}
\providecommand{\dom}{\mathop{\rm dom}\nolimits}
\providecommand{\gph}{\mathop{\rm gph}\nolimits}
\providecommand{\argmin}{\mathop{\rm argmin}}

\newcommand{\DS}[1]{{\color{black}{#1}}}
\newcommand{\DA}[1]{{\color{black}{#1}}}

\newcommand{\DAs}[1]{}%{\color{magenta}{\sout{#1}}}}
\newcommand{\DAss}[1]{}%{\color{magenta}{\cancel{#1}}}}
\begin{document}
	
	\begin{center}
		{\LARGE Cardinality Constraints in Single-Leader-Multi-Follower games}
		
		\vspace{1cm}
		
		{\Large \textsc{Didier Aussel, Daniel Lasluisa, David Salas}}
	\end{center}
	
	\bigskip
	
	\noindent\textbf{Abstract.} This work explores bilevel problems in the context of cardinality constraints. More specifically Single-Leader-Multi-Follower games (SLMFG) involving cardinality constraints are considered in two different configurations: one with the cardinality constraint at the leader's level and a mixed structure in which the cardinality constraint is split between leader and followers problem. We prove existence results in both cases and provided equivalent reformulations allowing the numerical treatment of these complex problems. The obtained results are illustrated thanks to an application to a facility location problem.
	
	\bigskip
	
	
	\vspace{0.6cm}
	
	
	%\tableofcontents
	
	%%===========================%%
	%%===========================%%

%%%%%%%%%%%%%%%%%%%%%%%%%%%%%%%%%%%%%%%%%%%%%%%%%%%%%%%%%%%%%%%%%%%%%%%%%%%%%%%%%%%%%%
%%%%%%%%%%%%%%%%%%%%%%%%%%%%%%%%%%%%%%%%%%%%%%%%%%%%%%%%%%%%%%%%%%%%%%%%%%%%%%%%%%%%%%
%%%%%%%%%%%%%%%%%%%%%%%%%%%%%%%%%%%%%%%%%%%%%%%%%%%%%%%%%%%%%%%%%%%%%%%%%%%%%%%%%%%%%%
%%%%%%%%%%%%%%%%%%%%%%%%%%%%%%%%%%%%%%%%%%%%%%%%%%%%%%%%%%%%%%%%%%%%%%%%%%%%%%%%%%%%%%
\section{Introduction}\label{sec:intro}
 Single-Leader-Multi-Follower (SLMF) games, \DA{a terminology introduced in \cite{PangFukushima2005} but games already considered in \cite{Stack}}, are bilevel games where one agent, the leader, interacts with a group of other agents, the followers, under a hierarchical structure. The leader \DA{takes} her decision variable $x\in\R^{p}$, \DA{and anticipates the reaction of} the followers \DAs{react} by solving a (generalized) Nash equilibrium problem, parametrized by $x$. \DAs{The leader, by knowing the equilibrium problem of the followers, anticipates this reaction and takes it into account during her decision process. }
 
 The games where only one follower is considered are known as Stackelberg games. They have been largely studied during the last decades, and their applications are now well spread along many interdisciplinary fields (see, e.g., \cite{sinha2017review,DempeZemkoho2020,BeckLjubicSchmidt2023,KleinertLabbeLjubicSchmidt2021} for some recent reviews and advances on the field). The case where multiple followers are involved is considerably more challenging, and therefore it still has a lot of open problems in comparison to its single-follower counterpart (see, e.g., \cite{AusselSvensson2020short,HuFukushima2015}).

A usual method \DAs{to try} to solve a Single-Leader-Multi-Follower game is to reformulate it as a single-level optimization problem, based on the natural variational formulation of the lower-level game played by the followers. Under standard convexity/continuity assumptions, one has that a vector $y\in\R^{q}$ is an equilibrium point for the lower-level game, if and only if it solves the coupled Karush-Kuhn-Tucker optimality conditions of each follower's optimization problem. The resulting problem fits into the class of \emph{Mathematical Programming with Complementarity Constraints (MPCC)}, since the complementarity equations naturally appear in KKT-based variational formulations. Recall that a complementarity constraint involving a vector variable $z\in\R^n$ is of the form
 \begin{equation}\label{eq:CompConstraint}
    T_1(z)\cdot T_2(z) = 0,
 \end{equation} 
 where $T_1,T_2:\R^{n}\to \R$ are two affine maps. As an illustration, the most common (but not the only) complementa\-rity constraint is given by complementary slackness in linear programming. 
 The MPCC reduction is classic in the literature of bilevel programming and it can be found in monographs like \cite{DempeEtAll2015,Dempe2002Foundations} for the case of one follower, and in \cite{AusselSvensson2020short} for the general SLMF game.
 
 In this work, we aim to study a particular class of SLMF games, involving the so-called \emph{cardinality constraints}. A cardinality constraint over a vector variable $z\in \R^{n}$ has the form of
 \begin{equation}\label{eq:CardinalityConstraint}
 \| z \|_0 = |\{ i\in \{1,\ldots,n\}\, :\, z_i\neq 0  \}| \leq K,
 \end{equation}
 where $K$ is a positive (integer) constant and $|A|$ denotes the cardinality of the set $A$. The function $ \| \cdot \|_0$ is known as the $\ell_0$-norm (or $\ell_0$-pseudonorm) and it counts the number of nonzero entries of a vector. It is commonly used in mathematical programming to model sparsity. However, due to its structure, the $\ell_0$-norm is hard to deal with. While most common approaches in the literature consist in replacing the $\ell_0$-norm by an alternative with more regularity properties (such as the $\ell_1$-norm), recent studies have tackled mathematical programming problems involving the $\ell_0$-norm directly (see, e.g.,  \cite{feng2013complementarity,ChancelierDeLara2021,ChancelierDeLara2022}).
 
 A very important contribution to deal with optimization problems with cardinality constraints was developed in \cite{BurdakovKanzowSchwartz2016,CervinkaKanzowSchwartz2016}, where the constraint \eqref{eq:CardinalityConstraint} was equivalently written as complementarity constraints as follows:
 \begin{equation}\label{eq:KeyReformulation}
 \|z\|_0\leq K \iff \exists u\in [0,1]^{n} \text{ such that }\left\{\begin{array}{l}\sum_{i=1}^n u_i \geq n-K,\text{ and }\\ \forall i\in\{1,\ldots,n \},\,u_iz_i = 0.\end{array}\right.
 \end{equation}
 This idea has been deeply exploited to provide constraints qualification, algorithms and first-order optimality conditions (see, e.g., \cite{FengEtAl2018,BucherSchwartz2018,KanzowRaharjaSchwartz2021,KrulikovskiEtAl2021,KrulikovskiEtAl2023,Ribeiro2022}). \DA{But this analysis has been systematically limited to single level optimization problems.}
 
Here, we will study SLMF games with cardinality constraints based on the following straightforward remark: the MPCC reformulation of SLMF games, and the reformulation \eqref{eq:KeyReformulation} of (single-level) optimization problems with cardinality constraints follow the same structure, namely, they encode the ``hard'' constraints by means of new variables (multipliers) and complementarity constraints.

The rest of the article is organized as follows: in Section \ref{sec:pre} we recall some results and fix some notation needed for the sequel. Section \ref{sec:upper} is devoted to study the case where the cardinality constraint acts as a coupling constraint in the leader's problem. We provide existence results and the analysis of MPCC single-level reformulations. 
%While the aforementioned case will be our main focus, we also provide in this section a simple example where cardinality constraints are considered as shared constraints for the followers, leading to serious infeasibility issues. 
In Section \ref{sec:mixed}, we profit from the complementarity constraints reformulations to provide an alternative model, by splitting the actions of the multipliers in \eqref{eq:KeyReformulation}: the leader decides the value of the multipliers, but the complementarity constraints are put at the followers' level. We show that this ``mixed'' version of the problem enjoys better existence results and similar single-level reformulations. Here, the split multipliers act as interdiction variables in an interdiction-like game (see, e.g., \cite{Fischetti2019Interdiction}). We finish this work by studying a variant of the well-known facility location problem \cite{marianov2003location,dan2019competitive}, where the leader must respect a cardinality constraint. 
 
%%%%%%%%%%%%%%%%%%%%%%%%%%%%%%%%%%%%%%%%%%%%%%%%%%%%%%%%%%%%%%%%%%%%%%%%%%%%%%%%%%%%%%
%%%%%%%%%%%%%%%%%%%%%%%%%%%%%%%%%%%%%%%%%%%%%%%%%%%%%%%%%%%%%%%%%%%%%%%%%%%%%%%%%%%%%%
%%%%%%%%%%%%%%%%%%%%%%%%%%%%%%%%%%%%%%%%%%%%%%%%%%%%%%%%%%%%%%%%%%%%%%%%%%%%%%%%%%%%%%
%%%%%%%%%%%%%%%%%%%%%%%%%%%%%%%%%%%%%%%%%%%%%%%%%%%%%%%%%%%%%%%%%%%%%%%%%%%%%%%%%%%%%%
\section{Preliminaties and notation \label{sec:pre}}
For any integer $n\in \N$, we write $[n] = \{1,\ldots,n\}$. From now on, we will work over finite-dimensional euclidean spaces $\R^n$, endowed with their respective usual inner products $\langle\cdot,\cdot\rangle$ and their induced euclidean norms. For any two vectors $a,b\in \R^n$, we write $a\odot b$ to denote \DAs{the} their Hadamard product, that is, $a\odot b = (a_ib_i~:~i\in [n]) \in\R^n$. 
For an extended-real valued function $f : \R^n \rightarrow \R \cup \{+\infty\}$, we denote by $\dom f$ the (effective) domain of $f$, that is, $\dom f = \{x\in\R^n~:~f(x)< +\infty\}$. For a given $\gamma \in \R$, we denote by $[f\leq \gamma]$ the sublevel set of $f$ of value $\gamma$. Recall that the function $f$ is said to be
\begin{itemize}
    \item \emph{convex} if $f(\lambda x + (1-\lambda)y) \leq \lambda f(x) + (1-\lambda)f(y)$ for each $x,y \in \dom(f)$ and $\lambda \in \DA{[0,1]}$.
    \item \emph{quasiconvex} if $f(\lambda x + (1-\lambda)y) \leq \max\{f(x),f(y)\}$, for each $x,y \in \dom(f)$ and $\lambda \in \DA{[0,1]}$.
\end{itemize}
It is well-known that every convex function is quasiconvex, and a function $f$ is quasiconvex if and only if the sublevel sets $[f\leq\gamma]$ are all convex (see, e.g., \cite{Aussel2014New}).

A function $h:\R^{n}\to \R$ is said to be \emph{weakly analytic} if for any two vectors $x,y\in\R^{n}$, the following implication holds:
\begin{equation}\label{eq:weaklyAnalyticImp}
    \begin{split}
        t\in\R\mapsto h(x+ty)&\text{ is constant over an open interval }\\
        &\implies h(x +ty) = h(x),\text{ for all }t\in \R.
    \end{split}
\end{equation}
In other words, $h$ is weakly analytic if, whenever it is constant over a segment, it must be constant over the whole line containing that segment. Of course, analytic function, such as affine functions, are weakly analytic (see, e.g., \cite{bank1983nonlinear}).

Let $X$ and $Y$ be two non-empty sets, and let us denote by $\mathcal{P}(Y)$ the power set of $Y$. A set-valued map (also known as correspondence) $F$ is a function $F : X \longrightarrow \mathcal{P}(Y)$, that is, a function which, for every $x\in X$, assigns a set $F(x) \subseteq Y$. We denote such a set-valued map as $F : X \rightrightarrows Y.$

Whenever $X$ and $Y$ are metric spaces, following \cite{AubinFrankowska2009Set}, we say that:
\begin{itemize}
    \item A set-valued map $F:X\tto Y$ is \emph{upper semicontinuous} at a point $x_0 \in X$ if, for each neighbourhood $V$ of $F(x_0)$ in $Y$, there exists a neighbourhood $U$ of $x_0$ in $X$ such that $$F(x)\subset V,\;\; \forall x \in U.$$
    
    \item A set-valued map $F:X\tto Y$ is \emph{lower semicontinuous} at a point $x_0 \in X$ if, for each open set $V \subset Y$ for which $F(x_0)\cap V \neq \emptyset $
    there exists a neighbourhood $U$ of $x_0$ such that 
    \[
    F(x)\cap V \neq \emptyset,\;\; \forall x \in U.
    \]
\end{itemize}

We define the domain of $F$ and the graph of $F$ as $\dom F = \{ x\in X~:~F(x)\neq\emptyset \}$, and $\gph F = \{ (x,y)\in X\times Y~:~y\in F(x) \}$, respectively. We say that $F$ is closed  if its graph is closed as a subset of $X\times Y$.

The following proposition, that will be used in the sequel, establishes that parametric sets given by separable weakly analytic constraints are lower semicontinuous (in fact, they enjoy further continuity properties, see, e.g., \cite[Theorem 4.3.5]{bank1983nonlinear} or \cite{SalasSvensson2023Bayesian}).
\begin{proposition} [{\cite[Theorem 3.3.3]{bank1983nonlinear}}]\label{prop:lsc-weaklyAnalytic} Let $X\subset \R^{n}$ be a closed set, and consider a set-valued map $F:X\tto \R^m$ given by 
\[
F(x) = \{ y\in \R^{m}~:~g_k(y) \leq \varphi_k(x),\,\forall k\in [m] \},
\]
where, for each $k\in[m]$, \DA{$\varphi_k:\R^n\to\R$ are continuous functions and the functions $\{g_k\}_{k\in[m]}$ are continuous} convex and weakly analytic, then the set-valued map $F$ is lower semicontinuous.
\end{proposition}

\paragraph{Constraint qualification for convex optimization problems:} Let us consider an abstract optimization problem
\begin{equation}\label{eq:Problemonelevel}
    \min_{z\in \Omega} f(z),
\end{equation}
\DA{where} the feasibility domain, $\Omega\subset \R^{n}$, \DS{ is a convex and closed set,} defined in terms of a set of constraints \DA{that is}
\begin{equation}\label{eq:General-constraintset}
    \Omega = \{z~:~g_k(z)\leq 0,\; k=1,\ldots,m \},
\end{equation}
\DA{with}, for each $k\in [m]$, the function $g_k:\R^{n}\to\R$ is \DS{(everywhere) differentiable}. \DAs{In this work, we focus only in convex smooth problems, that is, we will assume that objective function is of class $\mathcal{C}^1$ and convex, and all constraint functions are of class $\mathcal{C}^1$ and quasiconvex. In particular, this yields that the feasible sets $\Omega$ that we consider are always convex and closed.}

For $z\in\Omega$, we consider \textit{normal cone} of $\Omega$ at $z$ \DS{(in the sense of convex analysis)}, denoted by $N_{\Omega}(z)$, as
\begin{equation}\label{eq:normalConeDef}
N_{\Omega}(z)= \{ \nu\in\R^n~:~\langle \nu,y-z\rangle \leq 0,\,\forall y\in\Omega \}.
\end{equation}
We will say that the set $\Omega$ with its representation \eqref{eq:General-constraintset} verifies the \textit{Guignard's Constraint Qualification} \cite{guignard1969generalized} at $z\in \Omega$ if
\begin{equation}\label{eq:CQ}
    N_{\Omega}(z) = \left\{ \sum_{k=1}^{p} \lambda_k\nabla g_k(z)\, \Big|\ \lambda\odot g(z) = 0,\; \lambda\geq 0\right\}   
\end{equation}
There exist many sufficient conditions to ensure that $\Omega$ verifies \eqref{eq:CQ}. In particular, if all functions $g_k$ are affine, then  \eqref{eq:CQ} holds (see, e.g., \cite{solodov2010constraint}).

\paragraph{Single-Leader-Multi-Follower games:} We consider a game of $M+1$ agents with one leader and $M$ followers, where the latter are indexed by $i\in [M]$. The leader controls a variable $x \in \mathbb R^p$, while each follower $i\in [M]$ controls a variable $y_i \in \mathbb R^{q_i}$. We set $q = \sum_{i=1}^M q_i$.  We will write $y=(y_i,y_{-i}) \in \mathbb{R}^{q_i}\times \mathbb{R}^{q_{-i}} = \mathbb{R}^{q}$ to emphasize  the decision variable of follower $i$ within the aggregated decision vector $y$. Here, $y_{-i}$ stands for the aggregated decision vector of all followers, but follower $i$.

\begin{remark}\label{rem:Index} In what follows, we will reserve the index letter $i$ exclusively to denote the $i$th follower. Thus, for a vector $u\in\R^{q}$, we write $u_i$ to denote the vector in $\R^{q_i}$ corresponding to the coordinates in $u$ associated to the $i$th follower. To denote coordinates of a vector or a function, we will use other indexes, such as $k$ or $j$.
\end{remark}

 For the leader's problem, the objective function, which depends on \DA{both variables $x$ and $y$}, will be denoted by $\theta: \R^p \times \R^q \mapsto \R$. We denote by $X$ the feasible set for the leader's decision variable $x$. For each follower $i\in [M]$, following \eqref{eq:Problemonelevel}-\eqref{eq:General-constraintset}, \DAs{we denote by} $f_i: \mathbb R^p \times  \mathbb{R}^{q_i}\times \mathbb{R}^{q_{-i}}  \mapsto \mathbb R$ \DA{stands for her} cost function, and \DAs{by} $g_i: \mathbb R^p\times \mathbb R^{q_{i}}\times \mathbb{R}^{q_{-i}} \mapsto \mathbb R^{m_{i}}$ \DA{for her} constraint function, both functions depending on the variables of all players $(x, y_i,y_{-i})$. \DAs{We will write $p= \sum_{i=1}^M p_i$.} \DS{We will write $q = \sum_{i=1}^M q_i$.} We denote by $Y_i: X\times \DAss{\mathbb R^p\times} \mathbb R^{q_{-i}}\tto \mathbb R^{q_i}$  the feasibility set-valued map of follower $i$: given a leader's decision $x$, and the vector of decisions of the other followers, $y_{-i}$, the set $Y_i(x,y_{-i})$ is given by
\begin{equation}\label{eq:follower-feasibilityMap}
    (x,y_{-i})\longmapsto Y_{i}(x,y_{-i}) := \{ y_i\ : \ g_i(x,y_i,y_{-i})\leq 0 \}.
\end{equation}   

In the \DA{sequel} we will indistinctly use the notation $g_i(x,y_{i},y_{-i})\leq 0$ or $y_i\in Y_i(x,y_{-i})$ to refer to the constraints of the $i$th follower. The  general SLMF game  is defined in the following way:

\begin{equation}\label{eq:UpperLevel-Card_Preli}
    \begin{array}{|c|c|}
        \hline
        %&\\
        \begin{array}{cl}
            \displaystyle\min_{x,y}   & \theta(x,y)\\
            \text{s.t.} & \left\{
            \begin{array}{l}
                x \in X,\\
                G(x,y)\leq 0,\\
                y \in GNEP(x).
            \end{array}\right.
        \end{array}
        &
        \begin{array}{cl}
            \displaystyle\min_{y_i}   &f_i(x,y_i,y_{-i})\\
            \text{s.t.} & \DS{g_i(x,y_i,y_{-i})\leq 0.}%y_i \in Y_{i}(x,y_{-i}).
        \end{array}\\
        %&\\
        \hline
        \text{Leader}&\text{$i$th Follower}\\
        \hline
    \end{array}
\end{equation}
where $G:\R^{p}\times\R^{q}\to \R^r$ is a continuous function, and $GNEP(x)$ is the \DAs{solution} set-valued map of generalized Nash \DS{equilibriums} of the followers' \DS{problem} parametrized by $x$ (see, e.g. \cite{Fischetti2019Interdiction,PangFukushima2005}). Recall that the solution of the GNEP of the followers is given by
\begin{equation}\label{eq:GNEP-def}
 y\in GNEP(x) \iff \forall i\in [M], y_i\in \argmin\{ f_i(x,z,y_{-i})~:~z\in Y_i(x,y_{-i}) \}.
\end{equation}

In \eqref{eq:UpperLevel-Card_Preli}, the constraints $G(x,y)\leq 0$ are known as \emph{coupling constraints}, since they depend on both, upper-level and lower-level variables. They induce an extra difficulty for feasibility. It is the leader that must verify feasibility of $x\in X$ considering that
\begin{equation}\label{eq:FeasibilityLeader}
x\in X \text{ is feasible }\iff \{y~:~G(x,y)\leq 0\}\cap GNEP(x)\neq\emptyset.
\end{equation}
However, followers do not take into account the coupling constraint when they solve their equilibrium problem.

Finally, formulation  \eqref{eq:UpperLevel-Card_Preli} is known as optimistic, since the leader can choose the \DA{"better"} equilibrium point $y\in GNEP(x)$. Other formulations, such as the classic pessimistic approach (see, e.g. \cite{AusselSvensson2020short}) or the resent Bayesian approach (see \cite{SalasSvensson2023Bayesian}), are available in the literature, but they are out of the scope of the work.

%%%%%%%%%%%%%%%%%%%%%%%%%%%%%%%%%%%%%%%%%%%%%%%%%%%%%%%%%%%%%%%%%%%%%%%%%%%%%%%%%%%%%%
%%%%%%%%%%%%%%%%%%%%%%%%%%%%%%%%%%%%%%%%%%%%%%%%%%%%%%%%%%%%%%%%%%%%%%%%%%%%%%%%%%%%%%
%%%%%%%%%%%%%%%%%%%%%%%%%%%%%%%%%%%%%%%%%%%%%%%%%%%%%%%%%%%%%%%%%%%%%%%%%%%%%%%%%%%%%%
%%%%%%%%%%%%%%%%%%%%%%%%%%%%%%%%%%%%%%%%%%%%%%%%%%%%%%%%%%%%%%%%%%%%%%%%%%%%%%%%%%%%%%
\section{SLMF games with upper-level cardinality constraints \label{sec:upper}}
	
In this section we study the main object of our work: a SLMF game where the coupling constraint $G(x,y)\leq 0$ is a cardinality constraint. That is, we will study the problem

\begin{equation}\label{eq:UpperLevel-Card}
    \begin{array}{|c|c|}
        \hline
        %&\\
        \begin{array}{cl}
            \displaystyle\min_{x,y}   & \theta(x,y)\\
            \text{s.t.} & \left\{
            \begin{array}{l}
                x \in X,\\
                \norm{y} \leq K\\
                y \in GNEP(x)
            \end{array}\right.
        \end{array}
        &
        \begin{array}{cl}
            \displaystyle\min_{y_i}   &f_i(x,y_i,y_{-i})\\
            \text{s.t.} &g_i(x,y_i,y_{-i})\leq 0.% y_i \in Y_{i}(x,y_{-i}).
        \end{array}\\
        %&\\
        \hline
        \text{Leader}&\text{$i$th Follower}\\
        \hline
    \end{array}
\end{equation}

\begin{remark}
For the sake of simplicity, in this section we consider a single (global) cardinality constraint $	\norm{y} \leq K$. However, for any partition $\mathcal{S} = \{ s_1,\ldots,s_r\}$ of the involved index set $\{(i,j)~:~i\in~[M],~j\in [q_i]\}$, we can consider a sequence of ``independent'' cardinality constraints of the form
\begin{equation}
\norm{y_{s_l}} \leq K_l,\quad\forall l\in[r],
\end{equation}
where $y_{s_l} = (y_{i,j}~:~(i,j)\in s_l)$. The results of this section can be directly extended to this general case.
\end{remark}

\begin{remark}\label{rem:LowerLevelCardinality} One could also study the case where the cardinality constraint $\|y\|_0\leq K$ is put as a shared constraint for the followers' equilibrium problem, that is,
\begin{equation}\label{eq:UpperLevel-Card}
    \begin{array}{|c|c|}
        \hline
        %&\\
        \begin{array}{cl}
            \displaystyle\min_{x,y}   & \theta(x,y)\\
            \text{s.t.} & \left\{
            \begin{array}{l}
                x \in X,\\
                y \in GNEP(x)
            \end{array}\right.
        \end{array}
        &
        \begin{array}{cl}
            \displaystyle\min_{y_i}   &f_i(x,y_i,y_{-i})\\
            \text{s.t.} &\begin{cases}
            g_i(x,y_i,y_{-i})\leq 0,\\
                \norm{y} \leq K.
                \end{cases}
        \end{array}\\
        %&\\
        \hline
        \text{Leader}&\text{$i$th Follower}\\
        \hline
    \end{array}
\end{equation}
Such formulation is way more challenging, since the shared cardinality constraint might interfere with convexity and lower semicontinuity of the constraints maps of the followers. \DA{Such pathological behaviour has been revealed in \cite{Dan_PhD} but in} this first work, we will left this formulation out of the scope.
\end{remark}
%%%%%%%%%%%%%%%%%%%%%%%%%%%%%%%%%%%%%%%%%%%%%%%%%%%%%%%%%%%%%%%%%%%%%%%%%%%%%
\subsection{Existence of solutions}
The main problem with formulations with cardinality constraints in the upper level as in \eqref{eq:UpperLevel-Card}, is that the constraint $\norm{y}\leq K$ is a coupling constraint. This might lead to infeasibility, even if the equilibrium set $GNEP(x)$ is nonempty for every $x\in X$, as the following example shows.

\begin{example}[Infeasibility at the upper level problem by the cardinality constraint]
    \label{example_infeasibility}
    We con\-si\-der the leader's problem as 
    \begin{equation}\label{eq:exampUpperLeader}
        \begin{array}{cl}
            \displaystyle\min_{x,y} & x\\
            \text{s.t} & \left\{\begin{array}{l}      
                x \in [1,2]\\
                \norm{y} \leq 1\\
                y \in GNEP(x)
            \end{array}\right.
        \end{array}
    \end{equation}
    while the followers' equilibrium problem, for which $GNEP(x)$ is the solution set, is given by
    \begin{equation}\label{eq:exampleUpperLF}
        \begin{array}{|c|c|}
            \hline
            %&\\
            \begin{array}{cl}
                \displaystyle\min_{y_1}   & y_1\\
                \text{s.t.} & \left\{
                \begin{array}{l}
                    x \leq 2y_1\\
                    % y_1+y_2 \leq 3\\
                    y_1 \in [0,1]\\
                \end{array}\right.
            \end{array}
            &
            \begin{array}{cl}
                \displaystyle\min_{y_2}   & -y_2\\
                \text{s.t.} & \left\{
                \begin{array}{l}
                    y_2 \leq y_1\\
                    %y_1+y_2 \leq 3\\
                    y_2 \in [0,1]\\
                \end{array}\right.
            \end{array}\\
            %&\\
            \hline
            \text{Follower 1}&\text{Follower 2}\\
            \hline
        \end{array}
    \end{equation}
    
        It is not hard to check that for any $x\in [1,2]$, the followers' equilibrium problem has a unique solution given by $(y_1(x),y_2(x)) = (x/2,x/2)$. Thus, the solution map $x\mapsto GNEP(x)$ enjoys several amenable properties: it is single-valued, continuous, linear and nonempty for every leader's decision. However, the upper-level problem is infeasible. \hfill$\Diamond$
\end{example}

%%%%%%%%%%%%%%%%%%%%%%%%%%%%%%%%%%%%%%%%%%%%%%%%%%%%%%%%%%%%%%
%*********Existence Result***********\\
%%%%%%%%%%%%%%%%%%%%%%%%%%%%%%%%%%%%%%%%%%%%%%%%%%%%%%%%%%%%%%

Nevertheless, if one assumes that there exists at least one feasible point for the leader, one can replicate the standard existence result of \cite{AusselSvensson2018Some}.

\begin{theorem}\label{thm:ExistenceUpperLevel}
    Consider problem \eqref{eq:UpperLevel-Card} and assume that
    \begin{itemize}
        \item[(i)] $\theta$ is lower semicontinuous and $X$ is closed.
        \item[(ii)] for each $i\in [M]$, $f_i$ is continuous.
        \item[(iii)]   for each $i\in [M]$, the graphs of  $Y_i$ are uniformly bounded, and $X$ is bounded. 
        \item[(iv)] for each follower $i \in [M]$, the set-valued map $Y_i$
        is lower semicontinuous with nonempty closed graph.
    \end{itemize}
    Then, either the SLMF game \eqref{eq:UpperLevel-Card}  is infeasible, or it admits a solution.
\end{theorem}

\begin{proof}
    Let us assume that \eqref{eq:UpperLevel-Card}  is feasible, that is,  that the feasible set of the leader's problem
    \begin{equation}\label{F_feasible}
       \mathcal{F}:=\{(x,y) \in X\times\DA{\mathbb R^q} \;\mid\; \norm{y} \leq K,\; y \in GNEP(x) \}
    \end{equation}
    is nonempty.  We only need to show that in this case, there is a solution for \eqref{eq:UpperLevel-Card}. \DA{As an intermediate step}, we will first prove that the set-valued map GNEP has closed graph, thus defining a closed constraint set for the leader. To do so, we follow exactly the same technique as in \cite{AusselSvensson2018Some}, and the development \DA{are included} for the sake of completeness: Let us observe that we can write 
    \[
    GNEP(x) = \bigcap\limits_{i=1}^{M} S_{i}(x)
    \]
    with 
    \[
    S_i(x) := \big\{(y_i,y_{-i})  \;\mid\; y_i\in \text{argmin}_{z} \{f_i (x,z,y_{-i}) \;\mid\; z_i \in Y_i (x,y_{-i})\} \big\}.
    \]
    Thus it is sufficient to prove that each of the maps $S_i:X\tto \R^{q}$ have closed graph. Let us
    fix $i\in [M]$ and take sequences $(x_k)_k$  in $\mathbb R^p$ and $(y_k)_k$ in $\mathbb R^q$ converging respectively to $x$ and $y$, and such that $y_k  \in S_i(x_k)$ for all $k \in \mathbb N$. We want to prove that $y \in S_i(x)$. Note that
    \[
    (x_k,y_k) \in \gph S_i \implies y_{i,k}\in Y_{i}(x_k,y_{-i,k})\implies (x_k,y_k) \in \gph Y_i,
    \]
    and thus, since $Y_i$ has closed graph, we get that $y_i \in Y_i(x,y_{-i})$. Take $z_i \in  Y_i(x,y_{-i})$. By lower semicontinuity of the set-valued map $Y_i$, we know that \DAs{that} there exists $z_{i,k} \in Y_i(x_k,y_{-i,k})$ such that $z_{i,k} \rightarrow z_i$.
    Since $y_k \in S_i(x_k)$ then
    $$f_i(x_k,y_{i,k},y_{-i,k})\leq f_i(x_k,z_{i,k},y_{-i,k}), \forall k \in \mathbb N.$$
    Taking limit as $k\to\infty$,  continuity yields $f_i(x,y_{i},y_{-i})\leq f_i(x,z_{i},y_{-i}).$ Since $z_{i}$ was arbitrarily chosen from $Y_i(x, y_{-i})$, we conclude that $ y \in S_i(x)$. Thus $S_i(x)$ is closed and hence, GNEP has closed graph.
    
    Observe also that $GNEP$ is also uniformly bounded, since 
    \[
    GNEP(x) = \bigcap_{i=1}^{M}S_i(x)\subset \bigcap_{i=1}^M\{ y\in \mathbb R^{q}~:~y_i \in Y_{i}(x,y_{-i})\},
    \]
    and the right-hand set is uniformly bounded, thanks to hypothesis $(iii)$. Finally, note that
    \[
    \{(x,y)~:~x\in X,\, \norm{y}\leq K \} = \bigcup_{\tiny\begin{array}{c}J\subset [q],\\|J| = q-K\end{array}} \{(x,y)~:~x\in X,\, y_k = 0, \forall k\in J \}.
    \]
    Thus, the \DA{set} $\{(x,y)~:~x\in X,\, \norm{y}\leq K \}$ is closed as the finite union of closed sets, and so we get that the set
    \[
    \mathcal{F} = \{(x,y)~:~x\in X,\, \norm{y}\leq K \}\cap \mathrm{gph}(GNEP)
    \]
    is compact. Since the objective function $\theta$ is lower semicontinuous, it follows that the optimization problem of the leader in \eqref{eq:UpperLevel-Card} has a solution, by a mild application of Weierstrass theorem.\qed
\end{proof}

As usual, if the objective of the leader is coercive (that is, if $\theta(x,y) \to +\infty$ as $\|(x,y)\|\to +\infty$), then we can remove hypothesis $(iii)$ from Theorem \ref{thm:ExistenceUpperLevel}, and obtain the same result. Note also that, Theorem \ref{thm:ExistenceUpperLevel} coincides with \cite[Theorem 3.1]{AusselSvensson2018Some} in the case of trivial cardinality constraints (that is, $K = q$), and so, that is why the proof follows the same strategy.\\

An important feature of this result is that for many optimization problems, infeasibility can be checked numerically. Thus, if one can reformulate the SLMF game \DA{into} a single-level optimization problem, classic ready-to-go algorithms should be able to \DAs{to} decide infeasibility or to provide a (global/local) solution.  This property allows us to skip the step of checking existence, and pass directly to computation: either we will find the solution or we will get a certificate of infeasibility.
%%%%%%%%%%%%%%%%%%%%%%%%%%%%%%%%%%%%%%%%%%%%%%%%%%%%%%%%%%%%%%%%%%%%%%%%%%%%%%%%%%%%%%
%%%%%%%%%%%%%%%%%%%%%%%%%%%%%%%%%%%%%%%%%%%%%%%%%%%%%%%%%%%%%%%%%%%%%%%%%%%%%%%%%%%%%%
%%%%%%%%%%%%%%%%%%%%%%%%%%%%%%%%%%%%%%%%%%%%%%%%%%%%%%%%%%%%%%%%%%%%%%%%%%%%%%%%%%%%%%
%%%%%%%%%%%%%%%%%%%%%%%%%%%%%%%%%%%%%%%%%%%%%%%%%%%%%%%%%%%%%%%%%%%%%%%%%%%%%%%%%%%%%%

\subsection{Reformulations}\label{subsec:UpperLevel-Reformulations}

In this subsection, we consider that $GNEP(x)$ always stands for the solution set of the followers' equilibrium problem, as given in \eqref{eq:UpperLevel-Card}. Thus, we will only write the leader's problem, where the constraint $y\in GNEP(x)$ captures the interaction with the followers, as in \eqref{eq:GNEP-def}.\\

Now, following \cite{BurdakovKanzowSchwartz2016,CervinkaKanzowSchwartz2016}, we can rewrite the cardinality constraint of \eqref{eq:UpperLevel-Card}, by \DA{introducing} a new variable $u\in \R^{q}$ and using \eqref{eq:KeyReformulation} as follows:

\begin{equation}\label{eq:UpperLevel-Ref}
    \begin{array}{cl}
        \displaystyle\min_{x,y,u} & \theta(x,y)\\
        \text{s.t} & \left\{\begin{array}{l}      
            x \in X,\\
            y \in GNEP(x)\\
            \mathds{1}^\top u \geq q-K\\
            u\odot y = 0,\,\, u \in [0,1]^{q}\\
        \end{array}\right.
    \end{array}
\end{equation}

This reformulation is in fact equivalent (in the sense of global minimizers) to \eqref{eq:UpperLevel-Card}, as it is established in the following proposition. The proof is a direct adaptation of the proofs of \cite[Theorem 3.1, Theorem 3.4]{BurdakovKanzowSchwartz2016} \DA{and is mitted}. 

\begin{proposition}\label{prop:Equivalence-Card+Ref}    
    A vector $(x,y)\in \R^p\times\R^q$ is a feasible point (respectively, a global solution) for \eqref{eq:UpperLevel-Card} if and only if there exists a vector $u\in \R^q$ such that $(x,y,u)$ is a feasible point (respectively, a global solution) for \eqref{eq:UpperLevel-Ref}.
    \smallskip	
    Moreover, if $(x,y) \in \mathbb R^p \times \mathbb R^q$ is a local solution of \eqref{eq:UpperLevel-Card}, then there exists a vector $u \in \mathbb R^q$ such that the vector $(x,y,u)$ is also a local solution of \eqref{eq:UpperLevel-Ref}.
\end{proposition}

\begin{remark}\label{rem:binary-u}
Note that one can always choose $u\in \{0,1\}^q$ instead of $u\in[0,1]^q$ in \eqref{eq:UpperLevel-Ref}. Indeed, if $(x,y,u)$ is a feasible point (respectively, a global solution), one can define $\bar{u}\in\{0,1\}^q$ given by
\begin{equation}\label{eq:u_prob_L}
        \forall k \in[q],\quad \bar{u}_k=
        \begin{cases}
            0 & \text{if } u_k = 0,\\
            1 & \text{if } u_k \neq 0.
        \end{cases}
    \end{equation}
Then, it is straightforward to verify that $(x,y,\bar{u})$ is still feasible (respectively, a global solution) for the problem.
\end{remark}

Proposition \ref{prop:Equivalence-Card+Ref} is tight, in the sense that the converse implication for local minima doesn't hold. 
%Here, the local minima of problem \eqref{eq:UpperLevel-Ref} might fail to induce local minima of \eqref{eq:UpperLevel-Card}. 
The obstruction is that the variable $u$ in the reformulation \eqref{eq:UpperLevel-Ref} acts as a multiplier inducing partitions of the space: while the overall feasible set might be connected, it is possible to locally separate a point where some coordinate $u_k$ is strictly positive from those where it is zero. Similar issues have been identified in the classic MPCC reformulation of bilevel programming problems (see, e.g., \cite[Example 3.4]{DempeDutta2012}). The following example illustrates this fact.

\begin{example} Consider the problem 
    \[
    \begin{array}{cl}
        \displaystyle\min_{x,y} & y_1 - 2y_2\\
        \text{s.t} & \left\{\begin{array}{l}      
            x \in [0,1],\\
            y \in GNEP(x)\\
            \norm{y}\leq 1,
        \end{array}\right.
    \end{array}
    \]
and assume that for every 	$x\in [0,1]$, $GNEP(x) = [0,1]^2$. For any $x\in [0,1]$, the point $(x,0,1)$ is a global minimizer, while clearly $(x,0,0)$ is not a local optimum. Now, the reformulation \eqref{eq:UpperLevel-Ref} of this problem is given by
\[
\begin{array}{cl}
    \displaystyle\min_{x,y,u} & y_1 - 2y_2\\
    \text{s.t} & \left\{\begin{array}{l}      
        x \in [0,1],\\
        y \in GNEP(x)\\
        u_1+u_2 \geq 1\\
        u\odot y = 0,\,\, u \in [0,1]^{2}.\\
    \end{array}\right.
\end{array}
\]
Here, for any $x^*\in [0,1]$, the point $(x^*, y^*,u^*)$ with $y^* = (0,0)$ and $u^* = (0,1)$ is a local optimum. Indeed, fix \DAs{any} $r=1/2$ and pick any $(x,y,u)\in B_{r}(x, y^*,u^*)$ that is feasible for the reformulated problem. Then, necessarily $u_2 \geq 1/2$ and so $y_2 = 0$. Then,
\[
y_1 - 2y_2 = y_1 \geq 0 = y_1^* - 2y_2^*,
\] 
and the conclusion follows.\hfill$\Diamond$
\end{example}

\bigskip
%%%%%%%%%%%%%%%%%%%%%%%%%%%%%%%%%%%%%%%%%%%%%%%%%%%%%%%%%%%%%%
%*********MPCC***********\\
%%%%%%%%%%%%%%%%%%%%%%%%%%%%%%%%%%%%%%%%%%%%%%%%%%%%%%%%%%%%%%

The observation that the variable $u$ of \eqref{eq:UpperLevel-Ref} acts as a multiplier similar to the classic MPCC reformulation of SLMF games is in fact rather powerful. Indeed, we can profit from this observation to produce a second single-level reformulation of \eqref{eq:UpperLevel-Card} without cardinality constraints, which still is a mathematical programming problem with complementarity constraint. We will just replace the lower-level (generalized) Nash equilibrium problem of \eqref{eq:UpperLevel-Ref} by the concatenation of the associated parametric KKT conditions of each of the followers\DA{'s problem}. \\

%Recalling that the constraint set $Y_i(x,y_{-i})$ is given by functional inequalities as in \eqref{eq:follower-feasibilityMap}, 
We can consider the Lagrangian function for the $i$th follower as
\[
L_i(x, y, u,\lambda_i) := f_i(x,y,u) +\sum_{k=1}^{m_i}\lambda_{ik} g_{ik}(x,y,u),
\]
where $\lambda_i = (\lambda_{i,1},\ldots,\lambda_{i,m_i})$ stands for the vector of Lagrange multipliers. The MPCC reformulation for problem \eqref{eq:UpperLevel-Ref} is:

\begin{equation}\label{eq:UpperLevel-MPCC}
    \begin{array}{cl}
        \displaystyle\min_{x,y,u,\lambda} & \theta(x,y)\\
        \text{s.t.}  &\left\{\begin{array}{l}
            x\in X,\\
            \ind^\top u \geq q-K,\\
            u\odot y = 0,\, u \in [0,1]^q,\\
            \forall i\in[M],\, \left\{
            \begin{array}{l}   
                \nabla_{y_i} f_i(x,y) + \sum_{k=1}^{m_i}\lambda_{ik} \nabla_{y_i}g_{ik}(x,y) = 0,\\
                g_{i}(x,y)\leq 0,\\  
                \lambda_{i}\odot g_{i}(x,y) = 0, \\
                \lambda_{i}\geq 0. \\
            \end{array}
            \right.
        \end{array} \right.
    \end{array}
\end{equation}

Note that, since the lower-level problems of \eqref{eq:UpperLevel-Ref} are not affected by the variable $u$ (which is used only to rewrite the cardinality constraint), the same KKT equations from \eqref{eq:UpperLevel-MPCC} are used to provide a MPCC reformulation of \eqref{eq:UpperLevel-Card} maintaining \DA{the} cardinality constraint. However, we focus our attention only in \eqref{eq:UpperLevel-Ref} and  \eqref{eq:UpperLevel-MPCC}, since our main goal is to avoid cardinality constraints, in the spirit of \cite{BurdakovKanzowSchwartz2016,CervinkaKanzowSchwartz2016}.

We finish this subsection with the next theorem, which is one of our main results. It provides the equivalence of all problems we have written so far, in the sense of global solutions. Even though the proof follows standard arguments, we include it here for the sake of completeness.

\begin{theorem}\label{thm:Equiv_MPCC_Mixed} Consider problem \eqref{eq:UpperLevel-Card} and assume the following hypotheses:
    \begin{itemize}
        \item[($H_1$)] (Follower Differentiability) For any follower $i\in [M]$ and any $(x,y_{-i})\in X\times \mathbb R^{q_{-i}},$ $f_i(x,\cdot,y_{-i})$ and $g_i(x,\cdot,y_{-i})$ are differentiable;
        \item[($H_2$)] (Follower Convexity) For any follower $i\in[M]$ and any $(x,y_{-i})\in X\times \mathbb{R}^{q_{-i}},$ $f_i(x,\cdot,y_{-i})$ is convex, and the components of $g_i(x,\cdot,y_{-i})$ are quasiconvex functions;
        \item[($H_3$)] (Guignard's CQ) for each leader's strategy $x \in X$, for each follower $i \in [M]$, and for each joint strategy $y = (y_i, y_{-i})$ which is feasible for all followers, equation \eqref{eq:CQ} holds for $\Omega = Y_i(x,y_{-i})$ at $y_i$, with its representation \eqref{eq:follower-feasibilityMap}.  
    \end{itemize}
    Then, for any $(\bar{x},\bar{y})\in\R^{p}\times\R^{q}$, the following assertions are equivalent:
    \begin{itemize}
        \item[(i)] $(\bar{x}, \bar{y})$ is a feasible point (respectively, a global solution) of \eqref{eq:UpperLevel-Card}.
        \item[(ii)] $\exists \bar{u}\in [0,1]^q, (\bar{x}, \bar{y}, \bar{u})$ is a feasible point (respectively, a global solution) of \eqref{eq:UpperLevel-Ref}.
        \item[(iii)] $\exists \bar{\lambda}\in \mathbb R^m, \exists \bar{u}\in [0,1]^q, (\bar{x}, \bar{y}, \bar{u}, \bar{\lambda})$ is a feasible point (respectively, a global solution) of \eqref{eq:UpperLevel-MPCC}.
    \end{itemize}
\end{theorem}

\begin{proof}
    The equivalence between $(i)$ and $(ii)$ is given in Proposition \ref{prop:Equivalence-Card+Ref}. To show that $(ii)\iff (iii)$, it is enough to show that
    \[
    (x,y,u)\text{ is feasible for \eqref{eq:UpperLevel-Ref}}\iff \exists \lambda\in\R^m, (x,y,u,\lambda)\text{ is feasible for \eqref{eq:UpperLevel-MPCC}}.
    \]
    Suppose first that $(x,y,u)$ is feasible for \eqref{eq:UpperLevel-Ref}. Then, $y\in GNEP(x)$ and so we have that for every $i\in [M]$, $y_i \in \argmin_z\{ f_i(x,z,y_{-i})\mid z\in Y_i(x,y_{-i}) \}$. Since $Y_i(x,y_{-i})$ is convex, we get that 
    \[
    -\nabla_{y_i} f_i(x,y_i,y_{-i}) \in N_{Y_i(x,y_{-i})}(y_i).
    \]
    Then, hypothesis $(H_3)$ allows us to apply formula \eqref{eq:CQ} to $N_{Y_i(x,y_{-i})}(y_i)$, ensuring that there exists a multiplier $\lambda_i \in \R^{q_i}$ satisfying the Karush-Kuhn-Tucker conditions for the problem of the $i$th follower given by $(x,y_{-i})$, at $y_i$. Then, by writing $\lambda = (\lambda_1,\ldots, \lambda_M)$, we conclude that $(x,y,u,\lambda)$ is feasible for \eqref{eq:UpperLevel-MPCC}. 
    
    For the converse, suppose now that $(x,y,u,\lambda)$ is feasible for \eqref{eq:UpperLevel-MPCC}. Then, for $i\in [M]$, $\lambda_i \in \R^{q_i}$ is a multiplier satisfying the Karush-Kuhn-Tucker conditions for the problem of the $i$th follower given by $(x,y_{-i})$, at $y_i$. Let $z\in Y_i(x,y_i)$ and fix $k\in \{1,\ldots,m_i\}$. Since $g_{ik}(x,\cdot,y_{-i})$ is quasiconvex, we have that the segment $[y_i,z]$ is contained in the sublevel set $[g_{ik}(x,\cdot,y_{-i})\leq 0]$.  Then,
    \[
    \langle \nabla_{y_i} g_{ik}(x,y_i, y_{-i}),z-y_i\rangle = \lim_{t\to 0}\frac{g_{ik}(x,y_i + t(z-y_i), y_{-i}) - g_{ik}(x,y_i, y_{-i})}{t} \leq 0. 
    \]
    Since $z\in Y_i(x,y_{-i})$ and $k\in\{1,\ldots,m_i\}$ are arbitrary, we deduce that
    \[
    -\nabla_{y_i} f_i(x,y_i,y_{-i}) = \sum_{k=1}^{m_i}\lambda_{ik}\nabla_{y_i} g_{ik}(x,y_i,y_{-i}) \in N_{Y_i(x,y_{-i})}(y_i). 
    \]
    Thus, convexity of $f_i$ entails that $y_i \in \argmin_z\{ f_i(x,z,y_{-i})\mid z\in Y_i(x,y_{-i}) \}$. Since this holds for every $i\in [M]$, we conclude that $y\in GNEP(x)$, and so $(x,y,u)$ is feasible for \eqref{eq:UpperLevel-Ref}, finishing the proof.\qed
\end{proof}

\section{SLMF games with mixed cardinality constraints \label{sec:mixed}}
	
As illustrated in the Example \ref{example_infeasibility}, problems of the form \eqref{eq:UpperLevel-Card} can become infeasible due to the fact that the leader does not have control over the decisions of the followers. 
%Passing the cardinality constraints to the follower, as we discuss previously, can break the regularity of the upper-level problem (as illustrated in Example \ref{Example_Fail_lsc}), and thus it is not an option. 
In this section, however, we propose a mixed formulation considering the reformulation \eqref{eq:UpperLevel-Ref}, where the vector $u$ is used to represent the cardinality constraint partially distributed between the leader and the followers. Specifically, we propose to consider the following formulation:

\begin{equation}\label{eq:Mixed-Levels-Card}
    \begin{array}{|c|c|}
        \hline
        %&\\
        \begin{array}{cl}
            \displaystyle\min_{x,u,y}   & \theta(x,y)\\
            \text{s.t.} & \left\{
            \begin{array}{l}
                x \in X,\\
                y \in GNEP(x,u)\\
                \ind^\top u \geq q-K\\
                u \in [0,1]^q                  
            \end{array}\right.
        \end{array}
        &
        \begin{array}{cl}
            \displaystyle\min_{y_i}   &f_i(x,y_i,y_{-i})\\
            \text{s.t.} & g_{i}(x,y_i, y_{-i})\leq 0,\\
            & u_i\odot y_i = 0.
        \end{array}\\
        %&\\
        \hline
        \text{Leader}&\text{$i$th Follower}\\
        \hline
    \end{array}
\end{equation}

\medskip
 Recall that, according to Remark \ref{rem:Index}, the vector $u_i\in\R^{q_i}$ stands for the components of vector $u\in\R^q$ associated to the $i$th follower.  Consistently, the vector $u\in \R^{q}$ is written as $u=(u_i\,:\, i\in [M]) \in \prod_{i=1}^M\R^{q_i}$, since each part $u_i$ acts only on the $i$th follower's problem. Now, $GNEP(x,u)$ stands for the solution set of the new equilibrium problem of the followers. \\

Problem \eqref{eq:Mixed-Levels-Card} can be interpreted as follows: the leader is given a new interdiction variable $u\in[0,1]^q$, through which she can force any subset of followers' variables to be zero. The cardinality constraint is then ensured by demanding the leader to force at least $q-K$ followers' variables to be zero, or equivalently, to allow at most $K$ followers' variables to be nonzero. After the interdiction is decided by the leader, the followers solve their new equilibrium problem, respecting that any interdicted variable must be zero. 

%%%%%%%%%%%%%%%%%%%%%%%%%%%%%%%%%%%%%%%%%%%%%%%%%%%%%%%%%%%%%%%%%%%%%%%%%%%%%%%%%%%%%%
%%%%%%%%%%%%%%%%%%%%%%%%%%%%%%%%%%%%%%%%%%%%%%%%%%%%%%%%%%%%%%%%%%%%%%%%%%%%%%%%%%%%%%
	
\subsection{Existence of solutions and MPCC reformulation}

The main difficulty in order to be able to prove the existence of solutions for the mixed cardinality problem \eqref{eq:Mixed-Levels-Card}, is that due to the new constraints $u_i\odot y_i = 0$, the constraint maps of the followers are no longer lower semicontinuity. To solve this issue, we need to restrict ourselves to a setting where this difficulty doesn't occur.

Let us consider the case where the constraint functions $g_i(x, y_i,y_{-i}),$ for $i\in[M]$ can be written \DA{in the following separable form}
\[
g_i(x, y_i,y_{-i}) = \hat{g}_{i}(y_i) - \varphi_i(x,y_{-i}),
\]
where $\hat{g}_{i}: \R^{q_i}\to \R^{m_i}$ is componentwise convex, weakly analytic and continuous, and $\varphi_i:\R^p\times\R^{q_{-i}}\to\R^{m_i}$ is continuous. 

\begin{theorem}\label{thm:existence-Mixed}
    
    Consider problem \eqref{eq:Mixed-Levels-Card} and assume that
    
    \begin{itemize}
        \item[(i)] $\theta$ is lower semicontinuous and $X$ is closed;
        \item[(ii)] for each $i\in [M]$, $f_i$ is continuous;
        \item[(iii)]   for each $i\in [M]$, the images of $Y_i$,  are uniformly bounded, and $X$ is bounded;
        \item[(iv)]   for each $i\in [M]$, $g_i(x,y_i,y_{-i}) = \hat g_i(y_i) -\varphi_i(x,y_{-i})$, where $\hat g_i$  is componentwise convex, weakly analytic and continuous, and $\varphi_i$ is continuous. 
    \end{itemize}
Then, either problem \eqref{eq:Mixed-Levels-Card} is infeasible, or it admits a solution.
\end{theorem}

\begin{proof}
    
    Again, we will suppose that the feasible set of the leader's problem
            \begin{equation}\label{F_feasible_mixed}
                    \mathcal{F}:=\{(x,u,y) \in X\times[0,1]^q \times\mathbb R^q\;\mid\; \ind^\top u \geq q-K,\; y \in GNEP(x,u) \}
                \end{equation}
    is nonempty. So we only need to show that  \eqref{eq:Mixed-Levels-Card} admits a solution in this case.\\
    
    To do so, let $(x,u,y) \in X\times[0,1]^q\times \R^{q}$ be a feasible point for  \eqref{eq:Mixed-Levels-Card}. Note that taking $\bar{u}$ as
    \begin{equation}\label{eq:Normalization-u}
        \forall k\in[q],\,\bar{u}_k:=
        \begin{cases}
            0 & \text{if } u_k = 0\\
            1 & \text{if } u_k >0.  
        \end{cases}
    \end{equation}
    then $(x,\bar{u}, y)$ is feasible for the problem \eqref{eq:Mixed-Levels-Card} and the objective value for the leader is the same, since it doesn't depend on $u$. 
    %\DS{Moreover, $GNEP(x,u) = GNEP(x,\bar{u})$. Thus, if problem \eqref{eq:Mixed-Levels-Card} admits solutions, then there must be at least one solution $(x^*,u^*,y^*)$ with $u^*\in \{0,1\}^q$.} 
    
    Fixing  $u\in  \{0,1\}^{q} = \prod_{i=1}^{M} \{0,1\}^{q_i}$, we consider for every $i\in [M]$ the set 
    \begin{equation}\label{Set_linear_Zi_Mix}
        Z_i^u(x,y_{-i}) = Y_i(x,y_{-i}) \cap \{y_{i}\in \mathbb R^{m_i}\mid u_i\odot y_i=0\},
    \end{equation}
    and we define
    $Z^u(x,y) = \prod_{i=1}^{M} Z_i^u(x,y_{-i}).$ With these definitions, for every $u\in \{0,1\}^q$, we are going to define the following problem $(P_{u})$:

    \begin{equation}\label{eq:Mixed-Levels-Card_Zu}
        \begin{array}{|c|c|}
            \hline
            %&\\
            \begin{array}{cl}
                \displaystyle\min_{x,y}   & \theta(x,y)\\
                \text{s.t.} & \left\{
                \begin{array}{l}
                    x \in X_u = X\cap\text{dom}(Z^u),\\
                    y \in GNEP_u(x)               
                \end{array}\right.
            \end{array}
            &
            \begin{array}{cl}
                \displaystyle\min_{y_i}   &f_i(x,y_i,y_{-i})\\
                \text{s.t.} & y_i \in Z_i^u(x,y_{-i}).
            \end{array}\\
            %&\\
            \hline
            \text{Leader - ($P_u$)}&\text{$i$th Follower - ($P_u$)}\\
            \hline
        \end{array}
    \end{equation}
    \noindent \DA{where $GNEP_u(x)$ denotes the set of generalized Nash equilibriums for the game between the followers, as defined in \eqref{eq:Mixed-Levels-Card_Zu}, parametrized by $x$ and $u$.}\medskip\par
    Now, we will denote $(P)$ as the  problem \eqref{eq:Mixed-Levels-Card} and we consider\DA{, for any $u\in\{0,1\}^q$,} $v(P)$ and $v(P_u)$ the optimal values of problems $\eqref{eq:Mixed-Levels-Card}$ and \eqref{eq:Mixed-Levels-Card_Zu}, respectively. Here, we consider the convention that the value of an infeasible problem is $+\infty$. We claim that the \DA{value} $v(P)$ is equal to the minimal value among $\{v(P_u)~:~u\in\{0,1\}^q\}$. Indeed, on the one hand, if $(x,u,y)$ is feasible for $(P)$ with $u\in [0,1]^q$, then $(x,y)$ is feasible for $(P_{\bar{u}})$ where $\bar{u}$ is given as in \eqref{eq:Normalization-u}, and both points have the same objective value for their respective problems. So the value of problem $(P)$ satisfies 
    \[
    v(P) \geq \min\{v(P_u) \mid\ u\in \{0,1\}^q\}.
    \]
    
    On the other hand, if $(x,y)$ is feasible for $(P_u)$ for some $u\in\{ 0,1\}^{p}$, then $(x,y,u)$ must be feasible for $(P)$. Indeed, this follows directly by the observation that $y\in GNEP(x,u)$ if and only if $y\in GNEP_u(x)$. Thus,
    \[
    v(P) \leq \min\{v(P_u) \mid\ u\in \{0,1\}^q\},
    \]
    and the claim is proven. Now, in order to prove the existence of solutions for $(P)$ it is enough to show that for every $u\in\{0,1\}^{q}$, one has that $(P_u)$ is either infeasible or it admits a solution.  To do so, it is enough to verify that each problem $(P_u)$ verifies the hypotheses of \cite[Theorem 3.1]{AusselSvensson2018Some}, that is, it verifies the hypotheses $(i)-(iv)$ of Theorem \ref{thm:ExistenceUpperLevel}.
    
     Indeed, fix $u\in \{0,1\}^q$ and assume that \eqref{eq:Mixed-Levels-Card_Zu} is feasible. Trivially, hypotheses $(i)$ and $(ii)$ of Theorem \ref{thm:ExistenceUpperLevel}  hold. Note that, from $(iii)$, we have that the images $Y_i$ are uniformly bounded  and so the images of $Z_i^u$ are uniformly bounded, as well. Thus, $(P_u)$ also verifies hypothesis $(iii)$ of Theorem \ref{thm:ExistenceUpperLevel}. Moreover, we can write
    \[
    Z_i^u(x,y_{-}) := \{ z \: \ h_i(z) \leq \phi(x,y_{-i}) \},
    \]
    with 
     \[
     h_y(z) = \begin{pmatrix}
        \hat{g}_i(z)\\
        u_i\odot z\\
        -u_i\odot z
    \end{pmatrix} \text{ and }\phi(x,y_{-i})= \begin{pmatrix}
    \varphi_i(x,y_{-i})\\
    0\\
    0
    \end{pmatrix}.
    \]
    Since the mappings $y_i\mapsto u_i\odot y_i$ are componentwise linear with respect to $y_i$, we get that hypothesis $(iv)$ entails that $h_i$ is componentwise convex, weakly analytic and continuous and that $\phi$ is continuous. Thus, 
    from  \cite[Theorem 3.3.3]{bank1983nonlinear} (see Proposition \ref{prop:lsc-weaklyAnalytic} above), we have that the set-valued map $Z_i^u$ defined in \eqref{Set_linear_Zi_Mix} is lower semicontinuous and it has closed graph. Therefore, hypothesis $(iv)$ of Theorem \ref{thm:ExistenceUpperLevel} holds.
    
    Now, all the hypotheses of \cite[Theorem 3.1]{AusselSvensson2018Some} are satisfied, thus $(P_u)$ admits a solution. Since problem \eqref{eq:Mixed-Levels-Card} is feasible by hypothesis, at least one of the problems $(P_u)$ must be feasible as well. Then, to determine the solution of the problem \eqref{eq:Mixed-Levels-Card}, it is enough to take the best point $(x^*,y^*,u)$ with $(x^*,y^*)$ solving $(P_u)$, among the $u\in\{0,1\}^{q}$ for which the problem $(P_u)$ is feasible. This \DA{completes} the proof.\qed
\end{proof}

%%%%%%%%%%%%%%%%%%%%%%%%%%%%%%%%%%%%%%%%%%%%%%%%%%%%%%%%%%%%%%%%%%%%%%%%%%%%%%%%%%%%%%
%%%%%%%%%%%%%%%%%%%%%%%%%%%%%%%%%%%%%%%%%%%%%%%%%%%%%%%%%%%%%%%%%%%%%%%%%%%%%%%%%%%%%%
\subsection{Single-level reformulations for the linear lower-level problems}
\DA{Similarly as we did in Subsection \ref{subsec:UpperLevel-Reformulations},  one could consider the MPCC reformulation of Problem \eqref{eq:Mixed-Levels-Card}, and try to replicate the equivalence result given by Theorem \ref{thm:Equiv_MPCC_Mixed}.

The main obstruction to do so is that the constraint qualification \eqref{eq:CQ} is required on the new feasible sets $\{ z~:~ g_i(x,z,y_{-i})\leq 0,\, u_i\odot z=0 \}$, which might be hard to verify, since it is not automatically inherited from the  original constraint sets $Y_i(x,y_{-i}) = \{ z~:~ g_i(x,z,y_{-i})\leq 0\}$. 
However, this problem can be easily solved in the special case where the lower-level is given by linear problems.}

We say that problem \eqref{eq:Mixed-Levels-Card} (or problem \eqref{eq:UpperLevel-Card}) 
%is linear or that it 
has lower-level linear data if
\begin{enumerate}
    %\item $X = \{x\,~:\,~Ax~\leq~b\}$,  for some matrix $A$ and a vector $b$ of appropriate dimensions
    %\item $\theta(x,y) = c^{\top}x + d^{\top}y$, for some vectors $c\in \R^{p}$ and $d\in\R^{q}$.
    \item For each $i\in [M]$, there exist matrices $B_i,C_i,D_i$ of appropriate dimensions and a vector $\gamma_i\in\R^{m_i}$ such that $g_i(x,y_i,y_{-i}) = B_ix + C_iy_i +D_iy_{-i} - \gamma_i$.
    \item  For each $i\in [M]$, there exists a function $\alpha_i: \R^{p}\times \R^{q_{-i}}\to \R^{q_i}$ and a vector $\beta_i\in \R^{q_i}$ such that $f_i(x,y_i,y_{-i}) = \alpha_i(x,y_{-i})^{\top}y_i + \beta_i^{\top}(y_i\odot y_i)$.
\end{enumerate}

Observe that we are admitting the followers' functions to be either linear with respect to $y_i$ (if $\beta_i = 0$) or quadratic, since $\beta_i^{\top}(y_i\odot y_i) = \sum_{j=1}^{q_i}\beta_{ij}y_{ij}^2$. The abuse of the word ``linear'' comes from the fact that the gradient map $y_{i}\mapsto \nabla_{y_i}f(x,y_i,y_{-i})$ will be affine if the map $\alpha_i$ is affine.\\

Note that in this case, the constraint sets $\{z \, :\, g_i(x,z,y_{-i})\leq 0, u_i\odot y_i = 0 \}$ verify the linear constraint qualification (see, e.g., \cite{solodov2010constraint}), since all constraints are linear. Thus, \eqref{eq:CQ} is verified. Then, the MPCC reformulation of problem \eqref{eq:Mixed-Levels-Card} is given by
%Furthermore, conditions $(H_1)$ and $(H_2)$ of Theorem \ref{thm:Equiv_MPCC_Mixed} are also directly satisfied. 
%and Corollary \ref{cor:Equiv_MPCC_L} applies: the linear problem \eqref{eq:Mixed-Levels-Card} is either infeasible or it admits a solution.\\
    \begin{equation}\label{eq:MixedBP-MPCC-linear}
    \begin{array}{cl}
        \displaystyle\min_{x,y,u,\lambda,\nu} & \theta(x,y)\\ %c^{\top}x + d^{\top}y \\
        \text{s.t.}  &\left\{\begin{array}{l}
            x\in X,\\
            \ind^\top u \geq q-K,\\
            u \in [0,1]^q,\\
            \forall i\in [M], \left\{
            \begin{array}{l}   
                \alpha_i(x,y_{-i}) + \sum_{k=1}^{m_i}C_i^{\top}\lambda_{ik} +   \nu_{i}\odot u_{i}= 0,\\  
                \lambda_{ik} (B_ix + C_iy_i +D_iy_{-i} - \beta_i)= 0, \\
                B_ix + C_iy_i +D_iy_{-i} \leq \beta_i,\\
                u_i\odot y_i=0, \\
                \lambda_{ik}\geq 0.
            \end{array}
            \right.
        \end{array} \right.
    \end{array}
\end{equation}

\DAs{Note that,} Assuming that all $\alpha_i$ functions are affine maps and that $X=\{x~:~Ax\leq b\}$, problem \eqref{eq:MixedBP-MPCC-linear} is almost a linear programming problem with complementarity constraints: besides the objective function $\theta$, the only constraint that doesn't fit in this setting is the one associated to the derivatives of the Lagrangian functions.

For Linear Programming problems with complementarity constraints, one could apply the usual Branch-and-Bound algorithm. In fact, complementarity constraints and binary variables can be treated in the same way using SOS1 constraints (see, e.g. \cite{KleinertSchmidt2023,aussel2023tutorial}), and so Branch-and-Bound also applies for Mixed-Integer Linear Programming problems with complementarity constraints. The same approach can be applied also if the objective function $\theta$ is convex.%nonlinear, \textcolor{red}{but it maintains some structure such as convexity.}

Motivated by this observation and  by the fact that, following Remark \ref{rem:binary-u}, one can force the variable $u\in [0,1]^{q}$ to be an integer vector in $\{0,1\}^q$, we establish the following proposition, which is the last reformulation of this work.

\begin{proposition}\label{prop:FinalMPCC-LinearMixed} Consider problem \eqref{eq:MixedBP-MPCC-linear}, and the alternative formulation
            \begin{equation}\label{eq:MixedBP-MPCC-linear-final}
        \begin{array}{cl}
            \displaystyle\min_{x,y,u,\lambda,\eta} & \theta(x,y)\\ %c^{\top}x + d^{\top}y \\
            \text{s.t.}  &\left\{\begin{array}{l}
                x\in X,\\
                \ind^\top u \geq q-K,\\
                u \in \{0,1\}^q,\\
                \forall i\in [M], \left\{
                \begin{array}{l}   
                    \alpha_i(x,y_{-i}) + \sum_{k=1}^{m_i}C_i^{\top}\lambda_{ik} +   \eta_i= 0,\\  
                    \lambda_{ik} (B_ix + C_iy_i +D_iy_{-i} - \beta_i) = 0, \\
                    B_ix + C_iy_i +D_iy_{-i} \leq \beta_i,\\
                    u_i\odot y_i=0, \\
                    \eta_i\odot (\ind - u_i) = 0,\\
                    \lambda_{ik}\geq 0.
                \end{array}
                \right.
            \end{array} \right.
        \end{array}
    \end{equation}
    
    Then, for every point $(\bar{x},\bar{y})\in \R^{p}\times \R^{q}$, the following assertions are equivalent:
    \begin{itemize}
        \item [(a)] $\exists u\in [0,1]^q, (\bar{x}, \bar{y}, u)$ is a feasible point (respectively, a global solution) of \eqref{eq:Mixed-Levels-Card}.
        \item[(b)] $\exists (u,\lambda,\nu)\in [0,1]^q\times\R^m_+\times\R^{q}$ such that $(\bar{x}, \bar{y},u,\lambda,\nu)$ is a feasible point (respectively, a global solution) of \eqref{eq:MixedBP-MPCC-linear}.
        \item[(c)] $\exists (u,\lambda,\eta)\in \{0,1\}^q\times\R^m_+\times\R^{q}$ such that $(\bar{x}, \bar{y},u,\lambda,\eta)$ is a feasible point (respectively, a global solution) of \eqref{eq:MixedBP-MPCC-linear-final}.
    \end{itemize}
   % In particular, if the functions $\alpha_i$ are affine and $X$ is a convex polytope, problem \eqref{eq:MixedBP-MPCC-linear-final} is a mixed-integer linear programming problem with complementarity constraints.
\end{proposition}
\begin{proof}
Equivalence $(a)\iff (b)$ follows the same arguments as in Theorem \ref{thm:Equiv_MPCC_Mixed}, by noting that hypotheses $(H_1)-(H_3)$ are directly fulfilled by the linear data in the lower level. We thus \DA{only} prove $(b)\iff (c)$.

We will only show the equivalence between the assertions concerning feasible points, since the equivalence between global solutions follows from this first equivalence and the fact that the objective function is the same on both problems, depending only on the $(x,y)$ variables.

\paragraph{\DA{$(b)\implies (c)$}:} consider $\bar{u}$ as in \eqref{eq:Normalization-u}, that is: for each $k\in[q]$, $\bar{u}_k = 1$ if $u_k>0$, and $\bar{u}_k = 0$ if $u_k = 0$. Then, we can consider $\eta$ defined from $\nu$ as follows:
\[
\forall k\in[q],\,\eta_k = \begin{cases}
    \nu_k\cdot u_k\quad& \text{ if }u_k>0,\\
    0& \text{ if } u_k=0.
    \end{cases}
\]
Then, for every $k\in [q]$, we have that $\eta_k>0 \iff 1-\bar{u}_k = 0$, and that $\eta_k = \nu_k\cdot u_k$. With these observations, we can directly replace $u$ by $\bar{u}$ and $\nu_i\odot u_i$ by $\eta_i$ (for each follower $i\in [M]$) in problem \eqref{eq:MixedBP-MPCC-linear}, deducing that $(\bar{x},\bar{y},\bar{u},\lambda,\eta)$ is feasible for problem \eqref{eq:MixedBP-MPCC-linear-final}.

\paragraph{\DA{$(c)\implies (b)$}:} it is enough to define $\nu = \eta\odot u$, since the inclusion $u\in \{ 0,1\}^{q}$ and the constraints $\eta_i \odot (\ind-u_i) = 0$ allows to write:
\[
\nu_i\odot u_i = \eta_i\odot u_i \odot u_i = \eta_i, \quad \forall\,i\in[M].
\]
Replacing $\eta$ by $\nu\odot u$, the constraints $(\nu_i\odot u_i)\odot (\ind-u_i) =0$ become trivial in problem \eqref{eq:MixedBP-MPCC-linear-final}, and so we deduce that $(\bar{x},\bar{y},u,\lambda,\nu)$ is feasible for \DA{problem \eqref{eq:MixedBP-MPCC-linear}}. The proof is now completed.\qed
\end{proof}

%%%%%%%%%%%%%%%%%%%%%%%%%%%%%%%%%%%%%%%%%%%%%%%%%%%%%%%%%%%%%%%%%%%%%%%%%%%%%%%%%%%%%%
%%%%%%%%%%%%%%%%%%%%%%%%%%%%%%%%%%%%%%%%%%%%%%%%%%%%%%%%%%%%%%%%%%%%%%%%%%%%%%%%%%%%%%
%%%%%%%%%%%%%%%%%%%%%%%%%%%%%%%%%%%%%%%%%%%%%%%%%%%%%%%%%%%%%%%%%%%%%%%%%%%%%%%%%%%%%%
%%%%%%%%%%%%%%%%%%%%%%%%%%%%%%%%%%%%%%%%%%%%%%%%%%%%%%%%%%%%%%%%%%%%%%%%%%%%%%%%%%%%%%

\section{Application to Facility location problems with cardinality constraints \label{sec:app}}
	
In order to illustrate the \DA{applicability of the} theory \DA{in the sections above}\DAs{of  presented previous sections}, in this final section we provide \DA{now} a concrete example consisting in a variant of the \textit{Facility Location problem} \cite{marianov2003location,dan2019competitive}. This problem is a classic example in Mixed-Integer programming primers, and it is very relevant in several applications.  In particular, we focus on electric mobility and the need to optimize the charging infrastructure. Note that the model described below is an academic one and that made hypotheses could appear not too realistic but are assumed to avoid an useless complexification of the model since our aim is here to bring to the fore the effect/influence of the cardinality constraints. 

The problem we consider is the following: in a city, a social planner (the leader) is required to build charging stations for electric vehicles, within a (finite) list of strategic locations, indexed by the set $S$. The planner must take into account a set of drivers with electric cars (the followers), indexed by $i\in I$, \DA{who} will charge\DS{, once a week,} their vehicles on one \DA{or several} of the built stations\DAs{, recurrently}. Therefore, the \DA{planner} must decide whether to put or not a charging facility at each location $s\in S$, taking \DA{into account} the decision process of the \DA{drivers}. Here the planner is supposed to optimize social welfare, which considers the profit (or equivalently minimizing costs), as well as the benefit of the drivers. 

Due to the nature of the problem, we include some variants with respect to the classical formulation \DA{of the facility location problem}. First, the drivers are not necessarily committing to charge in the same station every time, but rather they might alternate within a set \DA{of} stations, with certain \DA{probabilities}. Second, the drivers play a \emph{congestion game}, in the sense that their preferences are influenced by how many drivers (in expectation) will \DA{choose} each of the stations. And finally, due to the characteristics of charging, a station can only serve a limited number of cars per day. With this limitation, the company must be able to satisfy the demand for any scenario induced by the distribution of the drivers. This last requirement is translated into a cardinality constraint. 

 %%%%%%%%%%%%%%%%%%%%%%%%%%%%%%%%%%%%%%%%%%%%%%%%%%%%%%%%%%%%%%%%%%%%%%%%%%%%%%%%%%%%%%
%%%%%%%%%%%%%%%%%%%%%%%%%%%%%%%%%%%%%%%%%%%%%%%%%%%%%%%%%%%%%%%%%%%%%%%%%%%%%%%%%%%%%%
		
\subsection{Upper-level and mixed formulations}\label{eq:examle_facility}
	
	In what follows, the index $i\in I$ will represent the $i$th driver, and the index $s\in S$, will be the $s$th station. For each station $s\in S$, we consider:
	
	\begin{itemize}
		\item a parameter $a_s>0$, which represents the total price of a full charge (vehicles are assumed to have all the same need);
		\item a parameter $c_s>0$, which represents the cost of installing the station $s$;
		\item a parameter $K_s\in\N$, which represents the number of charges that can be served by station $s$ per day;
		\item a variable $x_s \in \{0,1\}$, where $x_s = 1$ if the station is built, and $x_s=0$, otherwise;
        \item a factor $\alpha_{s}$ of inconvenience due to congestion at station $s$ (common for all drivers).
	\end{itemize}
	For each driver $i\in I$ and each station $s\in S$, we consider:
	\begin{itemize}
		\item a parameter $p_{is}$, representing the preference \DA{of driver $i$} to go to the station $s$. This preference is influenced by the price $a_s$, but also by implicit factors, such as the distance to home, the perception of the service, etc.;
		\item a variable $y_{is}\in [0,1]$, which determines probability of driver $i$ going to station $s$.
	\end{itemize}

We denote by $x = (x_s)_{s\in S}$ the decision vector of the company, by $y_i = (y_{is})_{s\in S}$ the decision vector of each driver $i\in I$ \DA{which represents the percentage of charge that driver $i$ will do at station $s$} \DS{in expectation}, and by $y = (y_{is})_{i\in I,\, s\in S}$ the joint decision vector of all drivers.  
\DA{Note for simplification of notations, we assume that all the cars need the same charge of energy, which is normalized at value 1 and this for each \DS{week} of the lifespan $T$.}

With this in mind, each driver $i\in I$ aims to maximize \DA{his (concave) satisfaction} function
\begin{equation}\label{eq:FLP-ObjectiveFollower}
y_{i}\mapsto f_i(x,y_{i},y_{-i}) = \sum_{s\in S} \left( p_{is} - \DS{\alpha_{s}}\sum_{j\in I} y_{js} \right)y_{is},
\end{equation}
\DAs{which is} parametrized by the planner's decision $x$ and by the other drivers' decision $y_{-i}$. The social planner aims to maximize the social welfare of the system, \DAs{which is} given by

\begin{equation}\label{eq:FLP-ObjectiveLeaderWelfare}
    \theta(x,y) = \sum_{s\in S}\left(T\sum_{i\in I} a_{s}y_{is} - c_sx_s\right) + T\sum_{i\in I} f_i(x,y).
\end{equation}
\DAs{where $T$ is the number of days considered as lifetime of the project.} The first term of $\theta(x,y)$ stands for the profit of the \DA{social planner} \DAs{suppliers} of each station and the second term is the cumulative benefit of the drivers. \DAs{As a simplifying assumption, we consider that all drivers charge their cars the same days (For example, once a week every Saturday).}

The formulation of the SLMF game with cardinality constraints \eqref{eq:UpperLevel-Card} is then given by

\begin{equation}\label{eq:Numerical_Example_upperL}
	\begin{array}{|c|c|}
		\hline
		%&\\
		\begin{array}{cl}
			\displaystyle\max_{x,y}   & \theta(x,y)\\
			\text{s.t.} & \left\{
			\begin{array}{l}
				y\in NEP(x),\\
				\norm{y_{\bullet,s}} \leq K_s, \forall s\in S\\
				x_{s} \in \{0,1\}, \forall s\in S.              
			\end{array}\right.
		\end{array}
		&
		\begin{array}{cl}
			\displaystyle\max_{y_{i,\bullet}}   &\DS{f_i(x,y_i,y_{-i})}\\%& \sum_{s\in S} \left( p_{is} - \alpha_{is}\sum_{j\in I} y_{js} \right)y_{is}\\
			\text{s.t.} & \sum_{s\in S} y_{is}=1,\\
			& y_{is}\leq x_{s},\;\;s\in S,\\
			& y_{is}\geq 0 \;\; s\in S.
		\end{array}\\
		%&\\
		\hline
		\text{Leader }&\text{$i$th Follower}\\
		\hline
	\end{array}
\end{equation}
\noindent \DA{where we can observe that $|S|$ cardinality constraints are considered. Observe that it could be interesting to include in the model coupling constraints on power/energy capacity at the scale of the EV charging stations. It is not done here since the aim is only to enlighten the influence of cardinality constraints on solutions of the Single-Leader-Follower problem.}
\medskip\par
Similarly, the SLMF game with mixed cardinality constraints \eqref{eq:Mixed-Levels-Card} is given by

	\begin{equation}\label{eq:Numerical_Example_mixed}
	\begin{array}{|c|c|}
		\hline
		%&\\
		\begin{array}{cl}
			\displaystyle\max_{x,y,u}   & \theta(x,y)\\
			\text{s.t.} & \left\{
			\begin{array}{l}
				y\in NEP(x,u),\\
				\ind^\top u_{\bullet,s} \geq p-K_s, \forall s\in S\\
				u_{is} \in [0,1], \forall i\in I,\, \forall s\in S\\
				x_{s} \in \{0,1\}, \forall s\in S.              
			\end{array}\right.
		\end{array}
		&
		\begin{array}{cl}
			\displaystyle\max_{y_{i,\bullet}}   &\DS{f_i(x,y_i,y_{-i})}\\%&\sum_{s\in S} \left( p_{is} - \alpha_{is}\sum_{j\in I} y_{js} \right)y_{is}\\
			\text{s.t.} & \sum_{s\in S} y_{is}=1,\\
			& y_{is}\leq x_{s},\;\;\forall s\in S\\
			& u_i\odot y_i=0, \\
			& y_{is}\geq 0, \;\;\forall s\in S.
		\end{array}\\
		%&\\
		\hline
		\text{Leader }&\text{$i$th Follower}\\
		\hline
	\end{array}
\end{equation}

\medskip
In both problems, we write $NEP$ instead of $GNEP$, to emphasize that the equilibrium problem of the followers \DA{game} is a Nash equilibrium problem and not a generalized one: the followers only affect each other through the objective function and not the constraints.

\DS{\begin{remark}\label{rem:Ljubic} The reader can observe that problem \eqref{eq:Numerical_Example_mixed} is equivalent to the classical facility location problem, but with cardinality constraints. Indeed, since there is no cost for the interdiction multiplier $u$, in practice the leader can force each follower to go to a designed facility, deciding at the same time the construction of facilities and the allocation of drivers. For the sake of illustration, we keep formulation \eqref{eq:Numerical_Example_mixed} that can be directly compared with the abstract forms developed in this work.
\end{remark}}

%%%%%%%%%%%%%%%%%%%%%%%%%%%%%%%%%%%%%%%%%%%%%%%%%%%%%%%%%%%%%%%
\subsection{Calibration of numerical experiments}

In this \DA{case} study  we consider \DS{$150$} drivers/cars and ten facilities $(s \in S= [10])$, each of which has a fixed capacity $K_s$, the $a_s$ cost  of fully charging the car battery at the facility $s$ (assuming all cars have the same battery).  

Each facility belongs exclusively to one type: cheap (type 1), medium (type 2) and expensive (type 3). The values of $a_s$ and $K_s$ depend only on the type.  Additionally, the fixed construction cost [USD], also depending on the type, is given by $c_1 = 10000$ for the cheap type, $c_2 = 30000$ for the medium type, and $c_3 = 50000$ for the expensive type. \DAs{We show} The parameters of each type of facility \DA{is given} in Table \ref{table:facilities_examp} \DA{while} \DAs{We show} the type of each facility \DA{is precised} in Table \ref{table:facilities_type}. Note that the values of the above constants, even if inspired from real life values, have been chosen to be able to enlighten the effect of cardinality constraints in this academic example. In the same line, as a simplifying assumption, we suppose that all cars charge their cars once a week on the same days. The lifespan $T$ of this study case, thus the number of days of charging, is fixed to 52 weeks, representing one year of operation.
	\begin{table}
		\centering
		\begin{tabular}{ |c|c|c|c|c|c|c|c|c|c|c| } 
			\hline
			         Facility&1&2&3&4&5&6&7&8&9&10\\
            \hline
                  Type&1&1&1&1&1&2&2&2&3&3\\
                  %\alpha$&0.11&0.1&0.1&0.11&0.1&0.1&0.1&0.09&0.11&0.1\\
			\hline
		\end{tabular}
		\caption{ Type of facilities.}
		\label{table:facilities_type}
	\end{table}
\begin{center}
	\begin{table}
		\centering
		\begin{tabular}{ |c|c|c|c|c| } 
			\hline
			 Type & Installation cost [USD] & Cost per charge [USD] & Capacity per day \\
			\hline
			1 &   $c_1 = 10000$  & $a_1 = 9$ & $K_1 = 20$\\
   			2 &   $c_2 = 30000$  & $a_2 = 15$ & $K_2 = 40$\\
      		3 &   $c_3 = 50000$  & $a_3 = 28$ & $K_3 = 100$\\
			\hline
		\end{tabular}
		\caption{ Description of facilities.}
		\label{table:facilities_examp}
	\end{table}
\end{center}
\begin{center}
\end{center}

For each facility, we consider  a random variable 
$\xi_s \sim \text{Lognormal(0.1,0.05)}$,
and  we set the inconvenience of the facility as \DS{$\alpha_{s} = 28\frac{\xi_s}{K_s}$}, for each facility $s\in S$. \DS{For this work, we consider the preference $p_{is}$ as follows: for each facility and each driver, we associate a random position (uniformly distributed) within the plane $[0,1]^2$, representing the locations of the facilities and the houses of the drivers. With this, we compute
\begin{equation}\label{eq:Preference}
    p_{is} = \frac{1}{d_{is}} - a_s ,
\end{equation}
where $d_{is}$ is the euclidean distance between the position of driver $i$ and the facility $s$.}

\DS{We performed 50 different experiments\DA{, here called ``cases",} where the values for $p$  and $\alpha$ were chosen randomly for each case. These values are included in the supplementary material. The problems under consideration are divided into two categories: those with cardinality constraints at the upper level associated to formulation \eqref{eq:Numerical_Example_upperL} (Upper configuration); and those with mixed cardinality constraints associated to formulation \eqref{eq:Numerical_Example_mixed} (Mixed configuration). We solve the given 100 problems (50 cases, 2 configurations per case), using \texttt{Gurobi v10.0.2} \cite{gurobi} \DA{coupled with the SOS1 package}, with \DA{an execution} time limit of 1 hour.}

\subsection{Results}

\DAs{In general} \DA{For the Mixed configurations problems}, the optimal solution was found within the time limit and with an acceptable optimality gap. \DAs{However} \DA{but}, for the upper configuration, the time limit was reached for many cases. Even worse, for about 20\% of cases it was not possible to find a feasible solution within the time limit. Table \ref{tab:summary} summarizes the performance values of the simulations. Figure \ref{fig:Gap_cases} provides the display of the optimality gap (given by \texttt{Gurobi}). The details can be found in the supplementary material.

\begin{table}[ht]
    \centering
    \begin{tabular}{|c|c|c|c|}
        \hline
        Configuration & Cases with feasible solution & Average Gap & Average Time [s]\\
        \hline
        %Upper-Profits & 41 & 0 &\\
         Upper & 36 & 0.079 &2199.386\\
        % Mixed-Profits & 50 & 0.005 & 0.47\\
         Mixed & 50 & 0.024 &84.844\\ 
         \hline
    \end{tabular}
    \caption{Summary of simulations. The average gap is taken considering only the cases where feasible points were found. The average time is taken over the 50 cases, in both configurations.}
    \label{tab:summary}
\end{table}

\begin{figure} [ht]	
	\centering
        %\captionsetup{type=figure}
	\includegraphics[width=.75\linewidth]{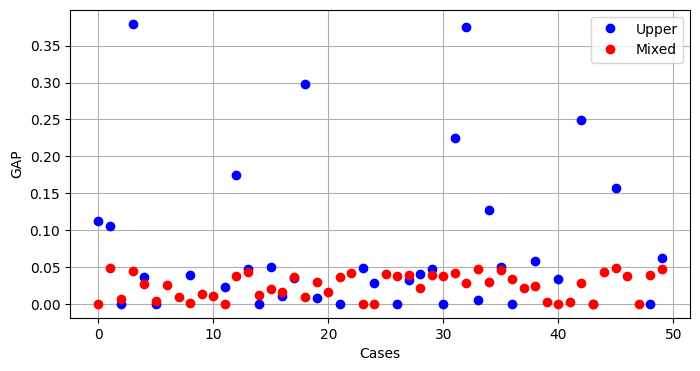}
	\caption{Optimality gap for each configuration in each case were at least one feasible solution was found. For the cases where no feasible solution was found, they are not displayed since the gap is $+\infty$. }
	\label{fig:Gap_cases}
\end{figure}

Figure \ref{fig:bulitfacilities_cases} shows two instances out of the 50 cases, emphasizing that under the parameters associated with these instances, an optimal solution was successfully obtained for both configurations. We can observe that for the Upper-Configuration, the optimization problem yields the same solution, which is to build two facilities with the highest service cost ($a_3$). \DA{Moreover one can remark that, in the Upper configuration, the optimal decision of the social planner does not depend on the case; see Fig. \ref{fig:bulitfacilities_cases} for cases 1 and 5 and supplementary material for the full set of cases.}

In contrast, for the Mixed-Configuration, the solution varies \DAs{between two possibilities} \DA{with cases}, considering \DA{the construction of three cheap and one expensive stations in case 1 and the construction of four cheap and two medium stations for case 5}. The details of all simulations can be found in the supplementary material.

\begin{figure}[ht]	
	\centering
       % \captionsetup{type=figure}   
	\includegraphics[width=.4\linewidth]{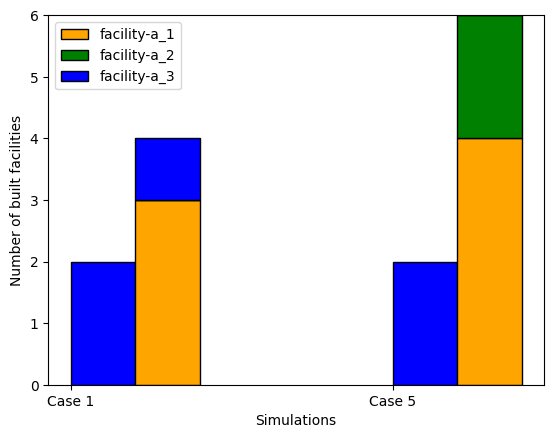}
	\caption{Each bar corresponds to the number of built facilities  of type $a_1$, $a_2$ or $a_3$, for \DAs{the different problems} \DA{cases 1 and 5}. The left bar corresponds to Upper configuration \DA{while} the right bar corresponds to Mixed configuration.}
	\label{fig:bulitfacilities_cases}
\end{figure}

Finally, we display the optimal values obtained in the simulations. Table \ref{tab:SummaryValues} \DA{presents} the average values of the profit of the supplier, the benefit of the followers, and the Social Welfare. Figure \ref{fig:ObjectiveFunc_cases} illustrates the distribution of objective function values for the leader. 

\begin{table}[ht]
    \centering
    \begin{tabular}{|c|c|c|c|}
    \hline
        Configuration & Profit of the Suppliers& Benefit of the followers & Social Welfare \\
        \hline
 %       Upper-Profit &&&\\
        Upper &118400.00&-186741.53&-68521.23\\
%        Mixed-Profit&&&\\
        Mixed &75129.44 &-108638.80&-34057.05\\
        \hline
    \end{tabular}
    \caption{Average values in [USD], considering for each configuration only the simulations with at least one feasible point. The Social Welfare is given by the sum of the profit of the leader and the benefit of the followers. }
    \label{tab:SummaryValues}
\end{table}

\begin{figure}	[ht]
\centering
        %\captionsetup{type=figure}
	\includegraphics[width=.5\linewidth]{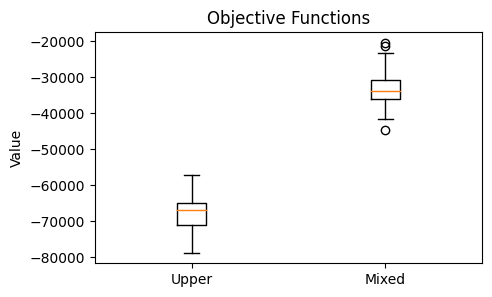}
	\caption{Boxplot of the values for the different objective functions of the leader, for each configuration. }
\label{fig:ObjectiveFunc_cases}
\end{figure}
Let us observe that\DA{, in all cases,} \DAs{it is not surprising that} the optimal social welfare is negative\DAs{: indeed} \DA{since} each driver is forced to charge his car and only the costs (for the driver) of charging the car is taken into account in the social welfare formula.

These results \DAs{show already} \DA{clearly enlighten} the difference between both configurations. The mixed formulation is, by construction, more flexible and it attains a better optimal value alternating between two optimal solutions: building one of each facilities, or building three cheap and one expensive. However, the upper formulation was trapped in the solution of building the two expensive facilities. Here we can appreciate \DA{the influence of} the restriction of the cardinality constraint: the suboptimal solution of building the two expensive facilities is chosen because \DA{is} the best one (and possibly the only one) that makes the cardinality constraints and the equilibrium constraints compatible.

While the results of optimization are substantially better for the mixed-configuration, \DA{simply endowing} a social planner with the interdiction-like multipliers is unrealistic. A very interesting perspective of this work is to study the facility location problem with a more complex model, including a mechanism to implement interdiction multipliers in such a way that drivers willingly comply to the planner decisions.

\section{Conclusions}

In this work, we \DA{focus} \DAs{delved} into Single-Leader-Multi-Follower problems involving  cardinality constraints.  There is limited literature on optimization with cardinality constraints \DA{but}, to our knowledge nothing on bilevel optimization with this kind of complex constraints; hence, we endeavored to develop both theoretical and practical results \DA{for this complex class of problems}.  First, we \DAs{demonstrated} \DA{proved} the existence of solutions for Single-Leader-Multi-Follower problems with cardinality constraints \DA{under mild assumption}, at the upper level. Additionally, equivalence results for global optimality were proven, coupled with reformulations of the original problem that facilitate numerical resolution.  
%In a second approach, namely SLMFG problems with cardinality constraints at the lower level, we observed a loss of convexity in the followers' constraint set, making it challenging to guarantee solution existence. 
Motivated by the obtained reformulations, an alternative approach for a Single-Leader-Multi-Follower problem with cardinality constraints was proposed, considering such constraints split \DA{between} both levels. In this latter case, solution existence was guaranteed for a specific scenario where follower constraint functions are componentwise convex, weakly analytic, and continuous functions. As \DA{an illustration, we propose formulation of the facility location problem for charging of electric vehicles based on a Single-Leader-Multi-Follower model} including cardinality constraints.

%%===========================%%
%%===========================%%
\bigskip
\textbf{Acknowledgements.} The first author benefits of the support of a FMJH-PGMO project. The second and third author have been partially supported by the Center of Mathematical Modeling CMM, grant FB210005 BASAL funds for centers of excellence (ANID-Chile).
The second author was partially supported by Beca Doctorado Nacional ANID-2020/21201478.
The third author was partially supported by the project FONDECYT 11220586 (ANID-Chile). The authors thank I. Ljubi\'{c} that pointed out Remark \ref{rem:Ljubic} to us during the PGMO days 2023. 
\bigskip

\textbf{Supplementary Material.} Available at github repository:\\ 
\url{https://github.com/dasalas22/Bilevel-CC-SupplementaryMaterial}.
\bibliographystyle{plain}
\bibliography{ref_CC}

    %%===========================%%
	%%===========================%%
 
    \noindent Didier AUSSEL
	
	\medskip
	
	\noindent CNRS PROMES, UPR 8521, Université de Perpignan Via Domitia \newline 
    Perpignan, France
	\smallskip
	
	\noindent E-mail: \texttt{aussel@univ-perp.fr}\newline\vspace{0.2cm}

    \noindent Daniel LASLUISA
	
	\medskip
	
	\noindent Departamento de Ingeniería Matemática, Universidad de Chile \newline Santiago, Chile
	\smallskip
	
	\noindent E-mail: \texttt{dlasluisa@dim.uchile.cl}\newline\vspace{0.2cm}

    \noindent David SALAS
	
	\medskip
	
	\noindent Instituto de Ciencias de la Ingenier\'{i}a, Universidad de
	O'Higgins\newline 
    Rancagua, Chile
	\smallskip
	
	\noindent E-mail: \texttt{david.salas@uoh.cl} \newline\noindent
	\texttt{http://davidsalasvidela.cl}

 \end{document}